\documentclass{amsart}%
\usepackage{amsmath}
\usepackage{amsfonts}
\usepackage{amssymb}
\usepackage{graphicx}%
\providecommand{\U}[1]{\protect\rule{.1in}{.1in}}
\newtheorem{theorem}{Theorem}
\theoremstyle{plain}

\newtheorem{corollary}{Corollary}

\newtheorem{lemma}{Lemma}

\newtheorem{proposition}{Proposition}
\newtheorem{remark}{Remark}

\numberwithin{equation}{section}
\begin{document}
\title[Uniqueness for some...]{Uniqueness for some classes of parabolic problems}
\author{F. Feo}
\address{{\tiny Dipartimento di Ingegneria}\\
{\tiny University of Naples Parthenope}\\
{\tiny Centro Direzionale, Isola C4 - 80143 Napoli}}
\email{{\tiny filomena.feo@uniarthenope.it}}
\date{}

\begin{abstract}
We prove some uniqueness results for weak solutions to some classes of
parabolic Dirichlet problems.

\end{abstract}
\maketitle

\numberwithin{equation}{section} \numberwithin{theorem}{section}
\numberwithin{lemma}{section} \numberwithin{remark}{section} \numberwithin{corollary}{section}\numberwithin{proposition}{section}

\section{Introduction}

In the present paper we investigate the uniqueness of weak solutions to the
following class of parabolic Dirichlet problems$\!$%
\begin{equation}
\!\!\left\{  \!%
\begin{array}
[c]{lll}%
\!u_{t}\!-\!\operatorname{div}\!a(x,t,u,\nabla u)\!+\!H(x,t,\nabla
u)\!+G(x,t,u)=\!f & \!\text{in}\! & \!Q_{T}\!:=\!\Omega\!\times\!\left(
0,T\right)  \!\!\\
\!u(x,t)\!=\!0 & \!\text{in}\! & \!\partial\Omega\!\times\!\left(  0,T\right)
\!\!\\
\!u(x,0)\!=\!u_{0}(x) & \!\text{on}\! & \!\Omega,\!\!
\end{array}
\right.  \label{prototipo}%
\end{equation}
where $\Omega$ is a bounded open subset of $\mathbb{R}^{N}$, $N\geq2$, $p>1$
and $T>0.$ We assume that $a:Q_{T}\times\mathbb{R\times R}^{N}\rightarrow
\mathbb{R}$, $H:Q_{T}\times\mathbb{R}^{N}\rightarrow\mathbb{R}$ and
$G:Q_{T}\times\mathbb{R}\rightarrow\mathbb{R}$ are Carath\'{e}odory functions
satisfying the following structural conditions%
\begin{equation}
a\left(  x,t,s,\xi\right)  \cdot\xi\geq\alpha_{1}\left\vert \xi\right\vert
^{p},\quad\label{ellittic}%
\end{equation}%
\begin{equation}
\left\vert a\left(  x,t,s,\xi\right)  \right\vert \leq\beta_{1}\left[
\left\vert s\right\vert ^{p-1}+\left\vert \xi\right\vert ^{p-1}+a_{1}%
(x,t)\right]  , \label{crescita}%
\end{equation}%
\begin{equation}
\left(  a\left(  x,t,s,\xi\right)  -a\left(  x,t,s,\xi^{\prime}\right)
\right)  \left(  \xi-\xi^{\prime}\right)  >0\text{ \ for }\xi\neq\xi^{\prime},
\label{monotonia}%
\end{equation}%
\begin{equation}
\left\vert H\left(  x,t,\xi\right)  \right\vert \leq b(x,t)\left\vert
\xi\right\vert ^{\gamma} \label{crescita H}%
\end{equation}
and
\begin{equation}
\left\vert G\left(  x,t,s\right)  \right\vert \leq c(x,t)\left\vert
s\right\vert ^{\lambda} \label{crescita G}%
\end{equation}
\noindent for a.e. $(x,t)\in Q_{T},$ $\forall s\in\mathbb{R},\forall\xi
,\xi^{\prime}\in\mathbb{R}^{N}$, where $\alpha_{1}$ and $\beta_{1}$ are
positive constants, $\gamma=p-\frac{N+p}{N+2},\lambda=p\frac{N+2}{N+1},$
$a_{1}\in L^{p^{\prime}}(Q_{T}),b\in L^{r}(Q_{T})$ and $c\in L^{\rho}(Q_{T})$
with $r=N+2$ and $\rho=\frac{N+p}{N}.$ Moreover $f=f_{0}-\operatorname{div}F$
with%
\begin{equation}
f_{0}\in L^{\left(  \frac{p(N+2)}{N}\right)  ^{\prime}}(Q_{T})\text{, }%
F\in\left(  L^{p^{\prime}}(Q_{T})\right)  ^{N}\text{ and }u_{0}\in
L^{2}(\Omega). \label{fff}%
\end{equation}
We recall that a weak solution\footnote{We refer to \cite{di benedetto} for
definitions of involved function spaces and parabolic framework.} to Problem
(\ref{prototipo}) is a measurable function belonging to $C(0,T;L^{2}%
(\Omega))\cap L^{p}(0,T;W_{0}^{1,p}(\Omega))$ such that $\forall t\in\left(
0,T\right]  $
\begin{align}
&  \int_{\Omega}uv(x,t)dx+%
{\displaystyle\iint_{Q_{t}}}
\left[  -uv_{t}+a(u,\nabla u)\nabla v+H(\nabla u)v+G(u)v\right]
dxd\tau\label{sol deb}\\
&  =\int_{\Omega}u_{0}v(x,0)dx+%
{\displaystyle\iint_{Q_{t}}}
\left(  f_{0}v+F\nabla v\right)  dxd\tau,\nonumber
\end{align}
for each $v\in W_{0}^{1,2}(0,T;L^{2}(\Omega))\cap L^{p}(0,T;W_{0}^{1,p}%
(\Omega)),$ where $Q_{t}:=\Omega\times\left(  0,t\right)  .$

If the function $a$ does not depend on $u$ and $H\equiv$ $G\equiv0,$ we known
there exists a unique weak solution (see \textit{e.g.} \cite{lions}). Here we
prove the uniqueness of weak solutions to Problem (\ref{prototipo}) when
$H\equiv0$ and $G\not \equiv 0$ assuming\ that operator $-\operatorname{div}%
a(x,t,u,\nabla u)$ is strongly monotone and the functions $a$ and $G$ are
Lipschitz continuous with respect to $u$ (that is usual\ as far as uniqueness
result concern)$.$ When $p\geq2$ we assume the principle part is not
degenerate, \textit{i.e.} in the model case $-\operatorname{div}a(x,t,u,\nabla
u)=-\operatorname{div}(a_{0}(x,t,u)(\varepsilon+\left\vert \nabla u\right\vert
^{p-1}\nabla u))$ with $\varepsilon>0$. In this case we can relax hypothesis
on the function $a$ assuming only a locally Lipschitz continuity with respect
to $u$ (see Section 2 for details). To our knowledge in literature there are
not any existence results for weak solutions to Problem (\ref{prototipo}) when
$H\equiv0$ and $G\not \equiv 0,$ then we will give some details (see
Proposition \ref{prop esistenza}) for convenience of the reader.

When $H\not \equiv 0$ and $G\equiv0,$ the existence of a weak solution to
problem (\ref{prototipo}) is investigated in \cite{Porzio}. If $a$ does not
depend on $u$ and under Lipschitz continuity on the lower order term $H$ , we
prove that such a solution is unique.

Our proof\ of uniqueness adapts the idea of \cite{alvino} (used also in
\cite{mio anisotropo} for an anisotropic elliptic operator) to the evolution
case: the main tool is the embedding in the parabolic equation framework. In
the case $H\equiv0$ and $G\not \equiv 0$ the method\ is improved using
Gronwall's Lemma. Our technique works also in proving some comparison principles.

There is an extensive literature about uniqueness of solution for elliptic
equations. We just mention some of these papers: e.g. \cite{alv-merc-fer},
\cite{alv-mercaldo}, \cite{alv-merc 2}, \cite{BP}, \cite{Bm}, \cite{B4 Cras},
\cite{Bocc}, \cite{casado}, \cite{Chipot} and \cite{GM}. For the evolution
case some uniqueness results can be found for example in \cite{artola},
\cite{russi}, \cite{l} and \cite{porzio 2} in the framework of weak solutions.
When datum $f$ is only integrable, uniqueness of renormalized and entropy
solutions is proved for example in \cite{andreu}, \cite{BM}, \cite{BMR}%
,\ \cite{mio con olivier uni} and \cite{prignet}.

\section{Statements of results}

First we study the case $H\equiv0$, that is we consider the following class of
nonlinear parabolic homogeneous Dirichlet problems%
\begin{equation}
\left\{
\begin{array}
[c]{lll}%
u_{t}-\operatorname{div}a(x,t,u,\nabla u)+G(x,t,u)=f & \text{in} & Q_{T}\\
u(x,t)=0 & \text{in} & \partial\Omega\times\left(  0,T\right) \\
u(x,0)=u_{0}(x) & \text{on} & \Omega,
\end{array}
\right.  \label{P0}%
\end{equation}
when (\ref{ellittic})-(\ref{monotonia}) and (\ref{crescita G})-(\ref{fff})
hold, function $a$\ also satisfies the following strong monotony condition
\begin{equation}
\left(  a\left(  x,t,s,\xi\right)  -a\left(  x,t,s,\xi^{\prime}\right)
\right)  \left(  \xi-\xi^{\prime}\right)  \geq\alpha\left(  \varepsilon
+\left\vert \xi\right\vert +\left\vert \xi^{\prime}\right\vert \right)
^{p-2}\left\vert \xi-\xi^{\prime}\right\vert ^{2} \label{monotonia forte}%
\end{equation}
and the following Lipschitz continuity condition
\begin{equation}
\left\vert a\left(  x,t,s,\xi\right)  -a\left(  x,t,s^{\prime},\xi\right)
\right\vert \leq\beta\left[  \phi+\left\vert \xi\right\vert ^{p-1}+\left(
\left\vert s\right\vert +\left\vert s^{\prime}\right\vert \right)  ^{\theta
}\right]  \left\vert s-s^{\prime}\right\vert \label{lip a}%
\end{equation}
and function $G$\ also satisfies the following Lipschitz continuity condition
\begin{equation}
\left\vert G\left(  x,t,s,\right)  -G\left(  x,t,s^{\prime}\right)
\right\vert \leq\varrho\left\vert s-s^{\prime}\right\vert \label{lip c}%
\end{equation}
for some $\theta\geq0$ with $\alpha>0$, $\varepsilon\geq0$, $\beta,\varrho>0$
and $\phi\geq0.$

We investigate separately the case $1<p<2$ and $p\geq2.$

\begin{theorem}
\label{th chipot}Let us assume $1<p<2,$ (\ref{ellittic})-(\ref{crescita}),
(\ref{crescita G})-(\ref{fff}), (\ref{monotonia forte}) with $\varepsilon=0$,
(\ref{lip a}) with $\theta=0$ and (\ref{lip c}) hold. Then there exists a
unique weak solution to Problem (\ref{P0}).
\end{theorem}

\begin{theorem}
\label{th p>2}Let us assume $p\geq2,$ (\ref{ellittic})-(\ref{crescita}),
(\ref{crescita G})-(\ref{fff}), (\ref{monotonia forte}) with $\varepsilon>0$,
(\ref{lip a}) with $0\leq\theta\leq\frac{p(N+2)}{2N}$ and (\ref{lip c}) hold.
Then there exists a unique weak solution to Problem (\ref{P0}).
\end{theorem}

\begin{remark}
Theorem \ref{th chipot} holds if we replace $\left\vert s-s^{\prime
}\right\vert $ in (\ref{lip a}) and (\ref{lip c}) by $\omega(\left\vert
s-s^{\prime}\right\vert ),$\ where $\omega:\left[  0,+\infty\right[
\rightarrow\left[  0,+\infty\right[  $ is such that $\omega(s)\leq s$ for
$0\leq s\leq\kappa$ for some $\kappa>0.$ Analogues generalization holds for
Theorem \ref{th p>2}.
\end{remark}

Moreover in order to prove uniqueness results for problems with lower order
term $H(x,t,\nabla u)$ we suppose function $a$ does not depend on $u$ and
$G\equiv0.$ More precisely we study the following class of nonlinear parabolic
homogeneous Dirichlet problems%
\begin{equation}
\left\{
\begin{array}
[c]{lll}%
u_{t}-\operatorname{div}a(x,t,\nabla u)+H(x,t,\nabla u)=f & \text{in} &
Q_{T}\\
u(x,t)=0 & \text{in} & \partial\Omega\times\left(  0,T\right) \\
u(x,0)=u_{0}(x) & \text{on} & \Omega,
\end{array}
\right.  \label{P 1}%
\end{equation}
when (\ref{ellittic})-(\ref{crescita H}) and (\ref{fff}) hold. We know (see
\cite{Porzio}) there exists at least a weak solution to Problem (\ref{P 1}).
As usual we assume the following locally Lipschitz condition on $H$%
\begin{equation}
\left\vert H(x,t,\xi)-H(x,t,\xi^{\prime})\right\vert \leq h(x,t)\left(
\eta+\left\vert \xi\right\vert +\left\vert \xi^{\prime}\right\vert \right)
^{\sigma}\left\vert \xi-\xi^{\prime}\right\vert \label{lip H}%
\end{equation}
with $\sigma\in%
\mathbb{R}
$, $\eta\geq0$ and $h$ a suitable function.

We investigate separately the case $p<2$ and $p\geq2.$

\begin{theorem}
\label{Th_Hp<2}Let us assume $\frac{2N}{N+2}\leq p<2$, (\ref{ellittic}%
)-(\ref{crescita}), (\ref{crescita H}), (\ref{fff}), (\ref{monotonia forte})
with $\varepsilon=0$ and (\ref{lip H}) with $h\in L^{\infty}(Q_{T}),$ $\eta>0$
and $\sigma\leq\frac{p-2}{2}$ hold. Then there exists a unique weak solution
to Problem (\ref{P 1}).
\end{theorem}

\begin{theorem}
\label{Th_Hp>2}Let us assume (\ref{ellittic})-(\ref{crescita}),
(\ref{crescita H}), (\ref{fff}), (\ref{monotonia forte}) with $\varepsilon>0$
and (\ref{lip H}) with $h\in L^{r}(Q_{T})$ for $r\geq N+2,$ $\eta=0$ and
$0\leq\sigma\leq p\left(  \frac{1}{N+2}-\frac{1}{r}\right)  +\frac{p-2}{2}$
hold. Then for $2\leq p\leq\frac{2r(N+2)}{r(N+2)+2(N+2)-2r}$ there exists a
unique weak solution to Problem (\ref{P 1}).
\end{theorem}

\begin{remark}
If $p>\frac{2r(N+2)}{r(N+2)+2(N+2)-2r},$ then Theorem \ref{Th_Hp>2} \ holds
with $0\leq\sigma\leq\frac{p}{N+2}-\frac{p}{r}$ and $\frac{p-2}{2}\leq
\sigma\leq\frac{p}{N+2}-\frac{p}{r}+\frac{p-2}{2}.$
\end{remark}

\begin{remark}
\label{remark c} If in Problem (\ref{P 1}) we add extra term $G(x,t,u)$\ that
is an increasing function in the variable $u$, the uniqueness of weak
solutions can be proved under hypothesis of Theorems \ref{Th_Hp<2} and
\ref{Th_Hp>2}$.$
\end{remark}

\bigskip

The arguments used in the proofs of previous theorems allows us to obtain also
some comparison principles.

\begin{corollary}
\label{C1}\emph{(Comparison principle)} In the hypothesis of Theorems
\ref{th chipot} and \ref{th p>2}, let us assume u and v are two solutions to
Problem (\ref{P0}) such that $u(x,0)\leq v(x,0)$ a.e. in $\Omega.$ Then
$u_{1}\leq u_{2}$ a.e. in $Q_{T}.$
\end{corollary}

\begin{corollary}
\label{C2}\emph{(Comparison principle)} In the hypothesis of Theorems
\ref{Th_Hp<2} and \ref{Th_Hp>2}, let us assume u and v are two solutions to
Problem (\ref{P 1}) such that $u(x,0)\leq v(x,0)$ a.e. in $\Omega.$ Then
$u_{1}\leq u_{2}$ a.e. in $Q_{T}.$
\end{corollary}

\section{Operators with a zero order term}

In this section we study Problem (\ref{P0}) when (\ref{ellittic}%
)-(\ref{monotonia}) and (\ref{crescita G})-(\ref{fff}) hold.

\subsection{Some preliminary results}

In order to prove theorems of the previous section we need to recall the
following embedding in the parabolic framework.

\begin{lemma}
\label{G-N Lemma} (see Proposition 3.1 of \cite{di benedetto}) Let $u\in
L^{\infty}(0,T;L^{\rho}(\Omega))\cap L^{p}(0,T;W_{0}^{1,p}(\Omega))$ with
$p\geq1$ and $\rho\geq1.$ Then $u\in L^{q}(Q_{T})$ with $q=p\frac{N+\rho}{N}$
and there exists a constant $C_{p}$ that depends on $N$ and $p$ such that
\begin{equation}
\left\Vert u\right\Vert _{L^{q}(Q_{T})}\leq C_{p}\left(  \underset{0<t<T}%
{\sup}\left\Vert u(.,t)\right\Vert _{L^{\rho}(\Omega)}+\left\Vert \nabla
u\right\Vert _{L^{p}(Q_{T})}\right)  . \label{G-N}%
\end{equation}
Moreover it results
\begin{equation}%
{\displaystyle\iint_{Q_{T}}}
\left\vert u\right\vert ^{q}\leq C_{p}^{q}\left(  \underset{0<t<T}{\sup}%
\int_{\Omega}\left\vert u\right\vert ^{\rho}\right)  ^{\frac{p}{N}}%
{\displaystyle\iint_{Q_{T}}}
\left\vert \nabla u\right\vert ^{p}. \label{G-N 2}%
\end{equation}

\end{lemma}

Moreover in the proof of uniqueness\ result for Problem (\ref{P0}) we need the
following version of Gronwall lemma.

\begin{lemma}
Let $T>0$ and let $a,d\ $be non-decreasing functions belonging to $L_{loc}%
^{1}(%
\mathbb{R}
_{+}),$ $b\in L_{loc}^{\infty}(%
\mathbb{R}
_{+})$ and $z\in L_{loc}^{1}(%
\mathbb{R}
_{+})$ such that%
\[
z(t)\leq a(t)+d(t)\int_{0}^{t}b(s)z(s)ds\text{ \ for a.e. }t\in\left[
0,T\right]  ,
\]
then%
\begin{equation}
z(t)\leq a(t)\left[  1+d(t)\int_{0}^{t}b(s)ds\exp\left(  d(t)\int_{0}%
^{t}b(s)ds\right)  \right]  \text{ \ \ for }a.e.\text{ }t\in\left[
0,T\right]  .\text{\ } \label{gronw 2}%
\end{equation}

\end{lemma}

\subsection{Existence of a weak solution}

To our knowledge in literature there are not existence results for weak
solutions to Problem (\ref{P0}).

\begin{proposition}
\label{prop esistenza}Under assumptions (\ref{ellittic})-(\ref{monotonia}) and
(\ref{crescita G})-(\ref{fff}) there exists at least a weak solution $u\in
L^{\infty}(0,T;L^{2}(\Omega))\cap L^{p}(0,T;W_{0}^{1,p}(\Omega))$ to Problem
(\ref{P0}).
\end{proposition}

The proof is standard but for convenience of reader we will give here some
steps. We observe that the coercivity of the operator is guaranteed only if
the norm $\left\Vert c\right\Vert _{L^{\rho}(Q_{T})}$ is small enough. Then as
usual we consider the approximate problems
\begin{equation}
\left\{
\begin{array}
[c]{lll}%
\left(  u_{n}\right)  _{t}+L_{n}u_{n}=f_{n}-\operatorname{div}F & \text{in} &
Q_{T}\\
u_{n}(x,t)=0 & \text{in} & \partial\Omega\times\left(  0,T\right)  \\
u_{n}(x,0)=u_{0}(x) & \text{on} & \Omega,
\end{array}
\right.  \label{appro}%
\end{equation}
where $L_{n}u=-\operatorname{div}a(x,t,u,\nabla u)+G_{n}(x,t,u),$
$G_{n}(x,t,s)=T_{n}(G(x,t,s))$, $T_{n}$ is the truncation at level $\pm n,$
defined by%
\begin{equation}
T_{n}(s)=\max\left\{  -n,\min\left\{  n,s\right\}  \right\}
\label{troncaa def}%
\end{equation}
and $\left\{  f_{n}\right\}  _{n\in%
\mathbb{N}
}\subset L^{p^{\prime}}(Q_{T})$ such that with $f_{n}\rightarrow f_{0}$
strongly in $L^{q^{\prime}}(Q_{T})$ with $q=\frac{p(N+2)}{N}.$ The operator
$L_{n}$ is pseudomonotone and coercive, then (see \textit{e.g.} \cite{lions})
there exists a weak solution $u_{n}\in L^{\infty}(0,T;L^{2}(\Omega))\cap
L^{p}(0,T;W_{0}^{1,p}(\Omega)).$ The following a priori estimate of $u_{n}$\ holds.

\begin{lemma}
\label{lemma stima}Assume that (\ref{ellittic})-(\ref{monotonia}) and
(\ref{crescita G})-(\ref{fff}) hold. If $u_{n}$ is a weak solution to Problem
(\ref{appro}), then there exists a constant $C_{0}$ (depending on the data
that appear in the structure conditions but not on $n$) such that
\begin{equation}
\underset{0<t<T}{\sup}\left\Vert u_{n}\right\Vert _{L^{2}(\Omega)}+\left\Vert
\left\vert \nabla u\right\vert \right\Vert _{L^{p}(Q_{T})}\leq C_{0}.
\label{stima a priori}%
\end{equation}

\end{lemma}

\begin{proof}
Using $u_{n}$ as test function in Problem (\ref{appro}), under
assumptions\ (\ref{ellittic}) and (\ref{crescita G}) we obtain for
$t\in\left(  0,T\right)  $
\begin{align}
&  \frac{1}{2}\int_{\Omega}u_{n}^{2}(t)dx+\alpha_{1}%
{\displaystyle\iint_{Q_{t}}}
\left\vert \nabla u_{n}\right\vert ^{p}dxd\tau\leq\label{aaa}\\
&
{\displaystyle\iint_{Q_{t}}}
c\left\vert u_{n}\right\vert ^{\lambda+1}dxd\tau+\frac{1}{2}\int_{\Omega}%
u_{0}^{2}dx+%
{\displaystyle\iint_{Q_{t}}}
\left(  f_{n}u_{n}+F\nabla u_{n}\right)  dxd\tau.\nonumber
\end{align}
Using H\"{o}lder inequality, (\ref{G-N}) and Young inequality we have%
\begin{align}
\!\!%
{\displaystyle\iint_{Q_{t}}}
\!\!\left(  \!f_{0}u_{n}\!+\!F\nabla u_{n}\!\right)  dxd\tau\!  &  \leq
\!\frac{\alpha_{1}}{p}\!\left\Vert \left\vert \nabla u\right\vert \right\Vert
_{L^{p}(Q_{t})}^{p}\!+\!\kappa_{1}\!\left\Vert \left\vert F\right\vert
\right\Vert _{L^{p\prime}(Q_{t})}\!+\!\kappa_{2}\left\Vert f_{n}\right\Vert
_{L^{q\prime}(Q_{t})}^{\frac{p(N+2)}{p(N+1)-N}}\!\label{stima f}\\
&  +\kappa_{3}\left[  \int_{\Omega}u_{n}^{2}(t)dx+\left\Vert \left\vert \nabla
u\right\vert \right\Vert _{L^{p}(Q_{t})}^{p}\right] \nonumber
\end{align}
for some positive constant\ $\kappa_{1},\kappa_{2}$ and $\kappa_{3}$ with
$\kappa_{3}<\min\left\{  \frac{1}{2},\frac{\alpha_{1}}{p^{\prime}}\right\}  .$
Using (\ref{G-N}) and Young inequality, we get
\begin{align}%
{\displaystyle\iint_{Q_{t}}}
c(x,\tau)\left\vert u\right\vert ^{\lambda+1}dxd\tau &  \leq\left\Vert
c\right\Vert _{L^{\rho}(Q_{t})}\left(  \underset{0<\tau<t}{\sup}\int_{\Omega
}u_{n}^{2}(\tau)dx\right)  ^{\frac{p}{N\rho^{\prime}}}\left(
{\displaystyle\iint_{Q_{t}}}
\left\vert \nabla u_{n}\right\vert ^{p}dxd\tau\right)  ^{\frac{1}{\rho
^{\prime}}}\label{stima G}\\
&  \leq\kappa_{4}\left\Vert c\right\Vert _{L^{\rho}(Q_{t})}\left[
\underset{0<\tau<t}{\sup}\int_{\Omega}u_{n}^{2}(\tau)dx+%
{\displaystyle\iint_{Q_{t}}}
\left\vert \nabla u_{n}\right\vert ^{p}dxd\tau\right] \nonumber
\end{align}
for some positive constant\ $\kappa_{4}.$ Using (\ref{stima f}) and
(\ref{stima G}) in (\ref{aaa}) and taking the supremum on $\left(
0,t_{1}\right]  $ for some $t_{1}\leq T$ such that $\left\Vert c\right\Vert
_{L^{\rho}(Q_{t_{1}})}$ is small enough, we obtain
\begin{equation}
\underset{t\in\left(  0,t_{1}\right]  }{\sup}\int_{\Omega}u_{n}^{2}(t)dx+%
{\displaystyle\iint_{Q_{t_{1}}}}
\left\vert \nabla u_{n}\right\vert ^{p}dxdt\leq\kappa_{5}\left[  \int_{\Omega
}u_{0}^{2}dx+\left\Vert \left\vert F\right\vert \right\Vert _{L^{p\prime
}(Q_{t_{1}})}+1\right]  , \label{finale stima}%
\end{equation}
for some positive constant $\kappa_{5},$ since $\left\{  f_{n}\right\}  _{n\in%
\mathbb{N}
}$ is bounded in $L^{q\prime}(Q_{T}).$ In order to avoid the assumption on
smallness of the norm $\left\Vert c\right\Vert _{L^{\rho}(Q_{T})}$ we split
(see also \cite{Porzio} and \cite{mio olivier esist}) the interval $\left[
0,T\right]  $ in $M$ small subinterval $\left(  t_{i},t_{i+1}\right)  $ for
$i=0,...,M-1$ in such a way $\left\Vert c\right\Vert _{L^{\rho}(\Omega
\times\left(  t_{i},t_{i+1}\right)  )}$ is small enough. We are able to derive
an estimate like (\ref{finale stima}) for small cylinder $\Omega\times\left(
t_{i},t_{i+1}\right)  .$ Finally taking the sum of different iterations,
(\ref{stima a priori}) holds for the inter cylinder $Q_{T}$.
\end{proof}

\begin{proof}
[Proof of Proposition \ref{prop esistenza}.]By Lemma \ref{lemma stima} it
follows that $\left\{  u_{n}\right\}  _{n\in%
\mathbb{N}
}$ is bounded\ sequence of $L^{\infty}(0,T;L^{2}(\Omega))\cap L^{p}%
(0,T;W_{0}^{1,p}(\Omega))$. Then it is possible to proceed as in the proof of
Theorem 2.2 of \cite{Porzio} to pass to the limit in (\ref{appro}) and to
conclude the existence of at least a weak solution to Problem (\ref{P0}).

We explicitly write only the computation about the boundness of $G_{n}%
(x,t,u_{n})$ in $L^{q^{\prime}}(Q_{T}).$ Ended by (\ref{crescita G}),
H\"{o}lder inequality and (\ref{G-N}) we have%
\begin{align*}%
{\displaystyle\iint_{Q_{T}}}
\left\vert G_{n}(x,t,u_{n})\right\vert ^{q^{\prime}}dxdt  &  \leq%
{\displaystyle\iint_{Q_{T}}}
c(x,t)^{q^{\prime}}\left\vert u_{n}\right\vert ^{\lambda q^{\prime}}%
dxdt\leq\left\Vert c\right\Vert _{L^{\rho}(Q_{t})}^{1-\frac{\lambda q^{\prime
}}{q}}\left(
{\displaystyle\iint_{Q_{T}}}
u_{n}^{q}\right)  ^{\frac{\lambda q^{\prime}}{q}}\\
&  \leq\left\Vert c\right\Vert _{L^{\rho}(Q_{t})}^{1-\frac{\lambda
q}{q^{\prime}}}\left[  C_{p}\left(  \underset{0<t<T}{\sup}\left\Vert
u_{n}(.,t)\right\Vert _{L^{2}(\Omega)}+\left\Vert \nabla u_{n}\right\Vert
_{L^{p}(Q_{T})}\right)  \right]  ^{\lambda q^{\prime}}.
\end{align*}
By Lemma \ref{lemma stima} the norm of $G_{n}(x,t,u_{n})$ in $L^{q^{\prime}%
}(Q_{T})$ is bounded by a constant depending on the data that appear in the
structure conditions but not on $n.$
\end{proof}

\subsection{Uniqueness of weak solutions}

In this subsection we prove \textit{ab aburdo} the uniqueness of weak
solutions to Problem (\ref{P0}).

\begin{proof}
[Proof of uniqueness in the hypothesis of Theorem \ref{th chipot}]We argue by
contradiction. Let us assume that Problem (\ref{P0}) admits two different
solutions $u$ and $v$ and $D=\left\{  (x,t)\in Q_{T}:w>0\right\}  $ has
positive measure, where $w=u-v.$ Using $\varphi=\frac{T_{k}(w^{+})}{k}$ for
$k\in\left[  0,\underset{D}{\sup}w^{+}\right[  $as test function in the
difference of the equations (where $T_{k}(\cdot)$ is defined in
(\ref{troncaa def})), we obtain for $t\in(0,T)$
\[
\!\int_{\Omega}w\varphi\!+\!%
{\displaystyle\iint_{\!Q_{t}}}
\!\!\!\left\{  -w\varphi_{t}+\left[  a\left(  x,t,u,\nabla u\right)  -a\left(
x,t,v,\nabla v\right)  \right]  \nabla\varphi+\left[
c(x,t,u)\!-\!c(x,t,v)\right]  \varphi\right\}  \!=\!0.\!
\]
Let us denote $\Psi_{k}(s)=\int_{0}^{s}T_{k}(\sigma)d\sigma.$ We have that
\begin{equation}
\int_{\Omega}w\varphi-%
{\displaystyle\iint_{Q_{t}}}
w\varphi_{t}=\frac{1}{k}\int_{\Omega}\Psi_{k}(w^{+}(t))\text{ }
\label{int per parti}%
\end{equation}
\ for $k>0.$ By (\ref{monotonia forte}), (\ref{lip a}), (\ref{lip c})
(\ref{int per parti}) we get%
\begin{align}
&  \frac{1}{k^{2}}\int_{\Omega}\Psi_{k}(w^{+}(t))+\alpha%
{\displaystyle\iint_{Q_{t}\cap D_{k}}}
\frac{\left\vert \nabla\varphi\right\vert ^{2}}{\left(  \left\vert \nabla
u\right\vert +\left\vert \nabla v\right\vert \right)  ^{2-p}}\label{equ 3}\\
&  \leq\beta%
{\displaystyle\iint_{Q_{t}\cap D_{k}}}
\left(  \phi+\left\vert \nabla v\right\vert ^{p-1}\right)  \left\vert
\nabla\varphi\right\vert +\frac{\varrho}{k}%
{\displaystyle\iint_{Q_{t}}}
\left\vert w\right\vert \varphi,\nonumber
\end{align}
where $D_{k}=\left\{  (x,t)\in D:w^{+}<k\right\}  .$ Using Young inequality
with some $\delta>0$ it follows%
\begin{align}
\!%
{\displaystyle\iint_{Q_{t}\cap D_{k}}}
\!\!\left(  \phi+\left\vert \nabla v\right\vert ^{p-1}\right)  \left\vert
\nabla\varphi\right\vert  &  \leq\!\!\frac{\delta\left(  \phi+1\right)  }{2}\!%
{\displaystyle\iint_{Q_{t}\cap D_{k}}}
\!\frac{\left\vert \nabla\varphi\right\vert ^{2}}{\left(  \left\vert \nabla
u\right\vert +\left\vert \nabla v\right\vert \right)  ^{2-p}}%
\!\!\label{intermedia}\\
\!+  &  \frac{\phi}{4\delta}\!%
{\displaystyle\iint_{Q_{t}\cap D_{k}}}
\!\!\!\left(  \!\left\vert \nabla u\right\vert +\left\vert \nabla v\right\vert
\!\right)  ^{2-p}\!\!+\!\frac{1}{4\delta}\!%
{\displaystyle\iint_{Q_{t}\cap D_{k}}}
\!\!\left(  \!\left\vert \nabla u\right\vert +\left\vert \nabla v\right\vert
\!\right)  ^{p}\!.\nonumber
\end{align}
Choosing $\delta$ small enough, using (\ref{intermedia}) and Young inequality
in (\ref{eq3}) and noticing that $\Upsilon(s)=2\Psi_{k}(s)-sT_{k}(s)\geq0$ for
$s\geq0$ ( check that $\Upsilon(s)=0$ for $0\leq s\leq k$ and $\Upsilon
^{\prime}(s)\geq0),$ we obtain
\begin{align}
&  \frac{1}{k^{2}}\int_{\Omega}\Psi_{k}(w^{+}(t))+c_{1}%
{\displaystyle\iint_{Q_{t}\cap D_{k}}}
\frac{\left\vert \nabla\varphi\right\vert ^{2}}{\left(  \left\vert \nabla
u\right\vert +\left\vert \nabla v\right\vert \right)  ^{2-p}}\label{eq3333}\\
&  \leq c_{2}\left(  \left\vert D_{k}\cap Q_{t}\right\vert +%
{\displaystyle\iint_{Q_{t}\cap D_{k}}}
\left(  \left\vert \nabla u\right\vert +\left\vert \nabla v\right\vert
\right)  ^{p}+\frac{1}{k^{2}}%
{\displaystyle\iint_{Q_{t}}}
\Psi_{k}(w^{+})\right) \nonumber
\end{align}
for some positive constants $c_{1},c_{2}$ independent on $k.$

\noindent Using Gronwall inequality (\ref{gronw 2})\ and taking the supremum
on $t\in(0,T)$, we get
\begin{equation}
\frac{1}{k^{2}}\underset{t\in(0,T)}{\sup}\int_{\Omega}\Psi_{k}(w^{+}(t))\leq
c_{2}\left(  1+Te^{T}\right)  \left(  \left\vert D_{k}\right\vert +%
{\displaystyle\iint_{D_{k}}}
\left(  \left\vert \nabla u\right\vert +\left\vert \nabla v\right\vert
\right)  ^{p}\right)  . \label{1}%
\end{equation}
It is easy to check that%
\begin{equation}
\zeta_{1}(k):=\left[  \left\vert D_{k}\right\vert +%
{\displaystyle\iint_{D_{k}}}
\left(  \left\vert \nabla u\right\vert +\left\vert \nabla v\right\vert
\right)  ^{p}\right]  \rightarrow0 \label{lim 00}%
\end{equation}
when $k$ goes to zero, then%
\begin{equation}
\underset{k\rightarrow0}{\lim}\frac{1}{k^{2}}\underset{t\in(0,T)}{\sup}%
\int_{\Omega}\Psi_{k}(w^{+}(t))=0. \label{tempo zero}%
\end{equation}
Recalling that $\frac{1}{2}\left\vert T_{k}(s)\right\vert ^{2}\leq\Psi
_{k}(s),$ we have
\begin{equation}
\underset{k\rightarrow0}{\lim}\underset{t\in(0,T)}{\sup}\int_{\Omega
}\left\vert \varphi\right\vert ^{2}=0. \label{tempo zero 2}%
\end{equation}
Coming back to (\ref{eq3333}), taking the supremum and using (\ref{tempo zero}%
) and (\ref{lim 00}), we obtain%
\begin{equation}
\underset{k\rightarrow0}{\lim}%
{\displaystyle\iint_{D_{k}}}
\frac{\left\vert \nabla\varphi\right\vert ^{2}}{\left(  \left\vert \nabla
u\right\vert +\left\vert \nabla v\right\vert \right)  ^{2-p}}=0.
\label{lim ellittico}%
\end{equation}
Moreover inequality (\ref{G-N 2}) and H\"{o}lder inequality imply%
\begin{align}
C_{1}^{-\frac{N+2}{N}}\left\vert D\backslash D_{k}\right\vert  &  \leq
C_{1}^{-\frac{N+2}{N}}\left\Vert \varphi\right\Vert _{L^{\frac{N+2}{N}}\left(
D\right)  }^{\frac{N+2}{N}}\leq\left(  \underset{t\in\left(  0,T\right)
}{\sup}\int_{\Omega}\left\vert \varphi\right\vert ^{2}\right)  ^{\frac{1}{N}}%
{\displaystyle\iint_{D_{k}}}
\left\vert \nabla\varphi\right\vert =\nonumber\\
\!\!  &  \leq\!\!\left(  \!\!\underset{t\in\left(  0,T\right)  }{\sup}%
\!\int_{\Omega}\!\left\vert \varphi\right\vert ^{2}\!\!\right)  ^{\!\frac
{1}{N}}\!\!\left(  \!%
{\displaystyle\iint_{\!D_{k}}}
\!\frac{\left\vert \nabla\varphi\right\vert ^{2}}{\left(  \left\vert \nabla
u\right\vert +\left\vert \nabla v\right\vert \right)  ^{2-p}}\!\right)
^{\!\frac{1}{2}}\!\!\left(  \!\!%
{\displaystyle\iint_{\!D_{k}}}
\!\!\!\left(  \!\left\vert \nabla u\right\vert +\left\vert \nabla v\right\vert
\!\right)  ^{2-p}\!\right)  ^{\!\frac{1}{2}}\!.\!\nonumber
\end{align}
By (\ref{lim 00}), (\ref{tempo zero 2}) and (\ref{lim ellittico}) it follows%
\begin{equation}
\left\vert D\right\vert =\underset{k\rightarrow0}{\lim}\left\vert D\backslash
D_{k}\right\vert =0. \label{lim D}%
\end{equation}
To complete the proof it suffices to replace $u$ and $v$.
\end{proof}

\begin{proof}
[Proof of uniqueness in the hypothesis of Theorem \ref{th p>2}]Arguing as in
the previous proof and taking into account the following extra term%
\[%
{\displaystyle\iint_{Q_{t}\cap D_{k}}}
\left(  \left\vert u\right\vert +\left\vert v\right\vert \right)  ^{\theta
}\left\vert \nabla\varphi\right\vert ,
\]
we obtain%
\begin{align}
&  \frac{1}{k^{2}}\int_{\Omega}\Psi_{k}(w^{+}(t))+\alpha%
{\displaystyle\iint_{Q_{t}\cap D_{k}}}
\left(  \varepsilon+\left\vert \nabla u\right\vert +\left\vert \nabla
v\right\vert \right)  ^{p-2}\left\vert \nabla\varphi\right\vert ^{2}%
\label{eq1}\\
&  \leq\beta%
{\displaystyle\iint_{Q_{t}\cap D_{k}}}
\left[  \phi+\left\vert \nabla v\right\vert ^{p-1}+\left(  \left\vert
u\right\vert +\left\vert v\right\vert \right)  ^{\theta}\right]  \left\vert
\nabla\varphi\right\vert +\frac{2\varrho}{k^{2}}%
{\displaystyle\iint_{Q_{t}}}
\Psi_{k}(w^{+}).\nonumber
\end{align}
Using Young inequality with some $\delta>0,$ we have the analogue of
(\ref{intermedia}), i.e.%
\begin{align}
&  \!%
{\displaystyle\iint_{Q_{t}\cap D_{k}}}
\left[  \phi+\left\vert \nabla v\right\vert ^{p-1}+\left(  \left\vert
u\right\vert +\left\vert v\right\vert \right)  ^{\theta}\right]  \left\vert
\nabla\varphi\right\vert \label{bb}\\
&  \leq\frac{\delta}{2}\left[  \frac{(\phi+1)}{\varepsilon^{p-2}}+1\right]  \!%
{\displaystyle\iint_{Q_{t}\cap D_{k}}}
\!\left(  \varepsilon\!+\!\left\vert \nabla u\right\vert \!+\!\left\vert
\nabla v\right\vert \right)  ^{p-2}\!\left\vert \nabla\varphi\right\vert
^{2}\!\!\nonumber\\
&  \!+\!\frac{\phi}{4\delta}\left\vert Q_{t}\cap D_{k}\right\vert
\!+\!\frac{1}{4\delta}%
{\displaystyle\iint_{Q_{t}\cap D_{k}}}
\left\vert \nabla v\right\vert ^{p}\!+\!\frac{1}{4\delta}\!%
{\displaystyle\iint_{Q_{t}\cap D_{k}}}
\!\left(  \left\vert u\right\vert \!+\!\left\vert v\right\vert \right)
^{2\theta}.\nonumber
\end{align}
Choosing $\delta$ small enough in (\ref{bb}), inequality (\ref{eq1}) gets
\begin{align}
&  \frac{1}{k^{2}}\int_{\Omega}\Psi_{k}(w^{+}(t))+c_{1}%
{\displaystyle\iint_{Q_{t}\cap D_{k}}}
\left(  \varepsilon+\left\vert \nabla u\right\vert +\left\vert \nabla
v\right\vert \right)  ^{p-2}\left\vert \nabla\varphi\right\vert ^{2}%
\label{mm}\\
&  \leq c_{2}\left[  \left\vert D_{k}\right\vert +%
{\displaystyle\iint_{Q_{t}\cap D_{k}}}
\left\vert \nabla v\right\vert ^{p}+%
{\displaystyle\iint_{Q_{t}\cap D_{k}}}
\!\left(  \left\vert u\right\vert \!+\!\left\vert v\right\vert \right)
^{2\theta}+\frac{1}{k^{2}}%
{\displaystyle\iint_{Q_{t}}}
\Psi_{k}(w^{+})\right]  .\nonumber
\end{align}
for some positive constants $c_{1},c_{2}$ independent on $k$.$\ $We remark
that since $2\theta\leq\frac{p(N+2)}{N}$%
\begin{equation}
\zeta_{2}(k):=\left\vert D_{k}\right\vert +%
{\displaystyle\iint_{D_{k}}}
\left\vert \nabla v\right\vert ^{p}+%
{\displaystyle\iint_{D_{k}}}
\!\left(  \left\vert u\right\vert \!+\!\left\vert v\right\vert \right)
^{2\theta}\rightarrow0 \label{lim 0}%
\end{equation}
when $k$ goes to zero. Arguing as\ in the next proof we obtain
(\ref{tempo zero 2}) and%
\begin{equation}
\underset{k\rightarrow0}{\lim}%
{\displaystyle\iint_{D_{k}}}
\left\vert \nabla\varphi\right\vert ^{2}=0. \label{lim eell2}%
\end{equation}
Moreover inequality (\ref{G-N}) and Young inequality imply
\begin{equation}
C_{1}^{-1}\left\vert D\backslash D_{k}\right\vert ^{\frac{N}{N+2}}%
\leq\underset{t\in\left(  0,T\right)  }{\sup}\left(  \int_{\Omega}\left\vert
\varphi\right\vert ^{2}\right)  ^{\frac{1}{2}}+\frac{1}{2}%
{\displaystyle\iint_{D_{k}}}
\left\vert \nabla\varphi\right\vert ^{2}+\frac{1}{2}\left\vert D_{k}%
\right\vert . \label{DDD}%
\end{equation}
By (\ref{DDD}), (\ref{lim 00}), (\ref{tempo zero 2}) and (\ref{lim eell2}) it
follows (\ref{lim D}). Then the assert holds changing the role of $u$ and $v$.
\end{proof}

\begin{remark}
If $G(x,t,u)=c(x)u$ the previous proofs follow easier multiplying equation by
$\exp(tc(x))$ and using as test function $\varphi=\frac{T_{k}(\left(
u-v\right)  ^{+}\exp(tc(x)))}{k}.$
\end{remark}

\begin{proof}
[Proof of Corollary 2.1]We can argue as in the previous proofs of uniqueness
putting $w=\left(  u-v\right)  ^{+}$.
\end{proof}

\section{Operators with a first order term}

We known that there exists at least a weak solution to Problem (\ref{P 1}).
Here we have to prove only the uniqueness.

\begin{proof}
[Proof of Theorem \ref{Th_Hp<2}]Let us suppose that $u$ and $v$ are two weak
solutions\ to Problem (\ref{P 1}) belonging to $L^{\infty}(0,T,L^{2}%
(\Omega))\cap L^{p}(0,T,W_{0}^{1,p}(\Omega))$ such that $w=u-v>0$ in a subset
$D\subset Q_{T}$ with $\left\vert D\right\vert >0.$ Let us denote
\[
w_{k}=\left\{
\begin{array}
[c]{ccc}%
w^{+}-k &  & \text{ if }w^{+}>k\\
0 &  & \text{otherwise}%
\end{array}
\right.
\]
for $k\in\left(  0,\underset{D}{\sup}w^{+}\right)  .$ We use $w_{k}$ as test
function in the difference of the equation:%
\begin{gather}
\!\!\!\int_{\!\Omega}\!ww_{k}\!+\!%
{\displaystyle\iint_{Q_{t}}}
\!\!\left[  -w\left(  w_{k}\right)  _{t}\!+\!\left(  a\left(  x,t,u,\nabla
u\right)  \!-\!a\left(  x,t,v,\nabla v\right)  \right)  \nabla w_{k}\right]
\!\!\!\label{equazione2}\\
\!\!\leq%
{\displaystyle\iint_{Q_{t}}}
\left\vert H\left(  x,t,\nabla v\right)  -H\left(  x,t,\nabla u\right)
\right\vert w_{k}\nonumber
\end{gather}
for $t\in\left(  0,T\right)  .$ We observe that
\begin{equation}
\int_{\Omega}ww_{k}-%
{\displaystyle\iint_{Q_{t}}}
w\left(  w_{k}\right)  _{t}=\frac{1}{2}\int_{\Omega}w_{k}^{2}(t).
\label{derivata tempo}%
\end{equation}
Using (\ref{derivata tempo}), (\ref{monotonia forte}) and (\ref{lip H}) with
$h\in L^{\infty}(Q_{T}),$ $\eta>0$ and $\sigma\leq\frac{p-2}{2},$ inequality
(\ref{equazione2}) becomes%
\[
\frac{1}{2}\int_{\Omega}w_{k}^{2}(t)+\alpha%
{\displaystyle\iint_{Q_{t}}}
\frac{\left\vert \nabla w_{k}\right\vert ^{2}}{\left(  \left\vert \nabla
u\right\vert +\left\vert \nabla v\right\vert \right)  ^{2-p}}\leq h%
{\displaystyle\iint_{Q_{t}}}
\frac{\left\vert \nabla w_{k}\right\vert w_{k}}{\left(  \eta+\left\vert \nabla
u\right\vert +\left\vert \nabla v\right\vert \right)  ^{-\sigma}}.
\]
Taking the supremum on $t\in\left(  0,T\right)  $ we obtain%
\begin{equation}
\frac{1}{2}\underset{0<t<T}{\sup}\int_{\Omega}w_{k}^{2}+\alpha%
{\displaystyle\iint_{E_{k}}}
\frac{\left\vert \nabla w_{k}\right\vert ^{2}}{\left(  \left\vert \nabla
u\right\vert +\left\vert \nabla v\right\vert \right)  ^{2-p}}\leq h%
{\displaystyle\iint_{E_{k}}}
\frac{\left\vert \nabla w_{k}\right\vert w_{k}}{\left(  \eta+\left\vert \nabla
u\right\vert +\left\vert \nabla v\right\vert \right)  ^{-\sigma}},
\label{eq 4}%
\end{equation}
where $E_{k}=\left\{  (x,t)\in Q_{T}:k<w^{+}<\underset{D}{\sup}w\right\}  .$
Since $\sigma\leq\frac{p-2}{2},$ Young inequality gets
\begin{equation}%
{\displaystyle\iint_{E_{k}}}
\frac{\left\vert \nabla w_{k}\right\vert w_{k}}{\left(  \eta+\left\vert \nabla
u\right\vert +\left\vert \nabla v\right\vert \right)  ^{\sigma}}\leq
\frac{\delta}{2}%
{\displaystyle\iint_{E_{k}}}
\frac{\left\vert \nabla w_{k}\right\vert ^{2}}{\left(  \left\vert \nabla
u\right\vert +\left\vert \nabla v\right\vert \right)  ^{2-p}}+\frac{1}%
{4\delta\eta^{-2\sigma+p-2}}%
{\displaystyle\iint_{E_{k}}}
\left\vert w_{k}\right\vert ^{2} \label{int 4}%
\end{equation}
for $\delta>0.$ Putting (\ref{int 4}) in (\ref{eq 4}) and choosing $\delta$
small enough we have
\begin{equation}
\underset{0<t<T}{\sup}\int_{\Omega}w_{k}^{2}+%
{\displaystyle\iint_{E_{k}}}
\frac{\left\vert \nabla w_{k}\right\vert ^{2}}{\left(  \left\vert \nabla
u\right\vert +\left\vert \nabla v\right\vert \right)  ^{2-p}}\leq c%
{\displaystyle\iint_{E_{k}}}
\left\vert w_{k}\right\vert ^{2}, \label{quasi fine}%
\end{equation}
where $c$ is a positive constant independent on $k$. Using (\ref{G-N 2}),
H\"{o}lder and Young inequalities and (\ref{quasi fine}) we have%
\begin{align*}%
{\displaystyle\iint_{E_{k}}}
\left\vert w_{k}\right\vert ^{2}  &  \leq C_{\frac{2N}{N+2}}^{2}\left[
\underset{0<t<T}{\sup}\int_{\Omega}\left\vert w_{k}\right\vert ^{2}\right]
^{\frac{2}{N+2}}%
{\displaystyle\iint_{E_{k}}}
\left\vert \nabla w_{k}\right\vert ^{\frac{2N}{N+2}}\\
\leq C_{\frac{2N}{N+2}}^{2}  &  \left[  \underset{0<t<T}{\sup}\int_{\Omega
}\left\vert w_{k}\right\vert ^{2}\right]  ^{\frac{2}{N+2}}\left(
{\displaystyle\iint_{E_{k}}}
\frac{\left\vert \nabla w_{k}\right\vert ^{2}}{\left(  \left\vert \nabla
u\right\vert +\left\vert \nabla v\right\vert \right)  ^{2-p}}\right)
^{\frac{N}{N+2}}\\
&  \times\left(
{\displaystyle\iint_{E_{k}}}
\left(  \left\vert \nabla u\right\vert +\left\vert \nabla v\right\vert
\right)  ^{\left(  2-p\right)  \frac{N}{2}}\right)  ^{\frac{2}{N+2}}\\
&  \leq C_{\frac{2N}{N+2}}^{2}\left[  \frac{2}{N+2}\underset{0<t<T}{\sup}%
\int_{\Omega}\left\vert w_{k}\right\vert ^{2}+\frac{N}{N+2}%
{\displaystyle\iint_{E_{k}}}
\frac{\left\vert \nabla w_{k}\right\vert ^{2}}{\left(  \left\vert \nabla
u\right\vert +\left\vert \nabla v\right\vert \right)  ^{2-p}}\right] \\
&  \times\left(
{\displaystyle\iint_{E_{k}}}
\left(  \left\vert \nabla u\right\vert +\left\vert \nabla v\right\vert
\right)  ^{\left(  2-p\right)  \frac{N}{2}}\right)  ^{\frac{2}{N+2}}\\
&  \leq c\frac{C_{\frac{2N}{N+2}}^{2}N}{N+2}\left(
{\displaystyle\iint_{E_{k}}}
\left\vert w_{k}\right\vert ^{2}\right)  \left(
{\displaystyle\iint_{E_{k}}}
\left(  \left\vert \nabla u\right\vert +\left\vert \nabla v\right\vert
\right)  ^{\left(  2-p\right)  \frac{N}{2}}\right)  ^{\frac{2}{N+2}},
\end{align*}
where $C_{\frac{2N}{N+2}}$ is the constant in (\ref{G-N}). It easily follows
that%
\[
1\leq c\frac{C_{\frac{2N}{N+2}}^{2}N}{N+2}\left(
{\displaystyle\iint_{E_{k}}}
\left(  \left\vert \nabla u\right\vert +\left\vert \nabla v\right\vert
\right)  ^{\left(  2-p\right)  \frac{N}{2}}\right)  ^{\frac{2}{N+2}}.
\]
Since $p\geq\frac{2N}{N+2},$ the right-hand side goes to zero when $k$ goes to
$\underset{D}{\sup}w^{+}$, which is impossible. To complete the proof we have
to change the role of $u$ and $v$.
\end{proof}

\begin{proof}
[Proof of Theorem \ref{Th_Hp>2}]We argue as in the proof of Theorem
\ref{Th_Hp<2}, obtaining%
\begin{align}
&  \frac{1}{2}\underset{0<t<T}{\sup}\int_{\Omega}\left\vert w\right\vert
^{2}+\alpha%
{\displaystyle\iint_{E_{k}}}
\left\vert \nabla w_{k}\right\vert ^{2}\left(  \varepsilon+\left\vert \nabla
u\right\vert +\left\vert \nabla v\right\vert \right)  ^{p-2}\label{eq3}\\
&  \leq%
{\displaystyle\iint_{E_{k}}}
h\left(  \left\vert \nabla u\right\vert +\left\vert \nabla v\right\vert
\right)  ^{\sigma}\left\vert \nabla w_{k}\right\vert w_{k}.\nonumber
\end{align}
Let us suppose $\sigma\geq\frac{p-2}{2}.$ Using H\"{o}lder inequality and
inequality (\ref{G-N}) we have%
\begin{equation}%
{\displaystyle\iint_{E_{k}}}
h\left(  \left\vert \nabla u\right\vert +\left\vert \nabla v\right\vert
\right)  ^{\sigma}\left\vert \nabla w_{k}\right\vert w_{k}
\label{intermedia 3}%
\end{equation}%
\begin{align}
\text{ \ \ \ }  &  \leq\left(
{\displaystyle\iint_{E_{k}}}
\left(  h\left(  \left\vert \nabla u\right\vert +\left\vert \nabla
v\right\vert \right)  ^{\sigma-\frac{p-2}{2}}\right)  ^{N+2}\right)
^{\frac{1}{N+2}}\nonumber\\
&  \times\left(
{\displaystyle\iint_{E_{k}}}
\left\vert \nabla w_{k}\right\vert ^{2}\left(  \varepsilon+\left\vert \nabla
u\right\vert +\left\vert \nabla v\right\vert \right)  ^{p-2}\right)
^{\frac{1}{2}}\left(
{\displaystyle\iint_{E_{k}}}
\left\vert w_{k}\right\vert ^{\frac{2(N+2)}{N}}\right)  ^{\frac{N}{2(N+2)}%
}\nonumber\\
&  \leq C_{2}\left(
{\displaystyle\iint_{E_{k}}}
\left(  h\left(  \left\vert \nabla u\right\vert +\left\vert \nabla
v\right\vert \right)  ^{\sigma-\frac{p-2}{2}}\right)  ^{N+2}\right)
^{\frac{1}{N+2}}\left(
{\displaystyle\iint_{E_{k}}}
\left\vert \nabla w_{k}\right\vert ^{2}\left(  \varepsilon+\left\vert \nabla
u\right\vert +\left\vert \nabla v\right\vert \right)  ^{p-2}\right)
^{\frac{1}{2}}\nonumber\\
&  \times\left[  \underset{0<t<T}{\sup}\left(  \int_{\Omega}\left\vert
w_{k}\right\vert ^{2}\right)  ^{\frac{1}{2}}+\left(
{\displaystyle\iint_{E_{k}}}
\left\vert \nabla w_{k}\right\vert ^{2}\right)  ^{\frac{1}{2}}\right]
,\nonumber
\end{align}
where $C_{2}$ is the constant in (\ref{G-N}). Putting (\ref{intermedia 3}) in
(\ref{eq3}),\ by Young inequality and some easy computation we obtain%
\begin{align*}
\min\left\{  \frac{1}{2},\alpha\varepsilon^{p-2}\right\}   &  \left[
\underset{0<t<T}{\sup}\int_{\Omega}\left\vert w\right\vert ^{2}+%
{\displaystyle\iint_{E_{k}}}
\left\vert \nabla w_{k}\right\vert ^{2}\right] \\
&  \leq\sqrt{2}C_{2}\max\left\{  1,\frac{1}{\varepsilon^{\frac{p-2}{2}}%
}\right\}  \left(
{\displaystyle\iint_{E_{k}}}
\left(  h\left(  \left\vert \nabla u\right\vert +\left\vert \nabla
v\right\vert \right)  ^{\sigma-\frac{p-2}{2}}\right)  ^{N+2}\right)
^{\frac{1}{N+2}}\times\\
&  \left[  \underset{0<t<T}{\sup}\int_{\Omega}\left\vert w_{k}\right\vert
^{2}+%
{\displaystyle\iint_{E_{k}}}
\left\vert \nabla w_{k}\right\vert ^{2}\left(  \varepsilon+\left\vert \nabla
u\right\vert +\left\vert \nabla v\right\vert \right)  ^{p-2}\right]  ,
\end{align*}
\textit{i.e.}%
\begin{equation}
1\leq\frac{\sqrt{2}C_{2}\max\left\{  1,\frac{1}{\varepsilon^{\frac{p-2}{2}}%
}\right\}  }{\min\left\{  \frac{1}{2},\alpha\right\}  }\left(
{\displaystyle\iint_{E_{k}}}
\left(  h\left(  \left\vert \nabla u\right\vert +\left\vert \nabla
v\right\vert \right)  ^{\sigma-\frac{p-2}{2}}\right)  ^{N+2}\right)
^{\frac{1}{N+2}}. \label{finale}%
\end{equation}
Since $\frac{N+2}{r}+\frac{\left(  \sigma-\frac{p-2}{2}\right)  (N+2)}{p}%
\leq1,$ the right-hand side in (\ref{finale}) goes to zero when $k$ goes to
$\underset{D}{\sup}w$, which is impossible.

\noindent Conversely if $0\leq\sigma<\frac{p-2}{2}$ as before we have%
\begin{align}
&  \!%
{\displaystyle\iint_{E_{k}}}
\left(  \left\vert \nabla u\right\vert +\left\vert \nabla v\right\vert
\right)  ^{\sigma}\left\vert \nabla w_{k}\right\vert w_{k}%
\!\!\label{intermedia 2}\\
\!\!  &  \leq\!\left(  \!%
{\displaystyle\iint_{E_{k}}}
\!\left(  h\left(  \left\vert \nabla u\right\vert +\left\vert \nabla
v\right\vert \right)  ^{\sigma}\!\right)  \!^{N+2}\!\right)  \!^{\frac{1}%
{N+2}}\!\!\!\left(  \!\!%
{\displaystyle\iint_{E_{k}}}
\!\!\left\vert \nabla w_{k}\right\vert ^{2}\!\right)  ^{\frac{1}{2}%
}\!\!\!\left(  \!\!%
{\displaystyle\iint_{E_{k}}}
\!\!\left\vert w_{k}\right\vert ^{\frac{2(N+2)}{N}}\!\right)  \!^{\frac
{N}{2(N+2)}}\!\nonumber\\
\!\!  &  \leq\!C_{2}\left(
{\displaystyle\iint_{E_{k}}}
\left(  h\left(  \left\vert \nabla u\right\vert +\left\vert \nabla
v\right\vert \right)  ^{\sigma}\right)  ^{N+2}\right)  ^{\frac{1}{N+2}%
}\!\left(
{\displaystyle\iint_{E_{k}}}
\left\vert \nabla w_{k}\right\vert ^{2}\right)  ^{\frac{1}{2}}\!\nonumber\\
\!\!  &  \times\left[  \underset{0<t<T}{\sup}\left(  \int_{\Omega}\left\vert
w_{k}\right\vert ^{2}\right)  ^{\frac{1}{2}}+\left(
{\displaystyle\iint_{E_{k}}}
\left\vert \nabla w_{k}\right\vert ^{2}\right)  ^{\frac{1}{2}}\right]
.\!\!\nonumber
\end{align}
Then putting (\ref{intermedia 2}) in (\ref{eq3}) by Young inequality and some
easy computation we have%
\[
1\leq\frac{\sqrt{2}C_{2}}{\min\left\{  \frac{1}{2},\alpha\varepsilon
^{p-2}\right\}  }\left(
{\displaystyle\iint_{E_{k}}}
\left(  h\left(  \left\vert \nabla u\right\vert +\left\vert \nabla
v\right\vert \right)  ^{\sigma}\right)  ^{N+2}\right)  ^{\frac{2}{N+2}}.
\]
Since $\frac{N+2}{r}+\frac{\sigma(N+2)}{p}\leq1,$ the contradiction follows
again. Changing the role of $u$ and $v$, we complete the proof.
\end{proof}

\begin{proof}
[Proof of Corollary 2.2]We can argue as in the proof of Theorems \ref{Th_Hp<2}
and \ref{Th_Hp>2}, putting $w=\left(  u-v\right)  ^{+}$.
\end{proof}

\end{document}